\documentclass[11pt]{amsart}
\usepackage{graphicx,a4wide}
\usepackage{amssymb}
\usepackage{showlabels,hyperref,color}

\newtheorem{thm}{Theorem}[section]
\newtheorem{cor}[thm]{Corollary}
\newtheorem{lem}[thm]{Lemma}

\theoremstyle{definition}

\theoremstyle{remark}

\numberwithin{equation}{section}

\newcommand{\IR}{\mathbb R}
\newcommand{\IN}{\mathbb N}
\newcommand{\Ra}{\Rightarrow}

\newcommand{\w}{\omega}
\newcommand{\e}{\varepsilon}

\begin{document} \sloppy

\title[Linearly continuous functions and $F_\sigma$-measurability]{Linearly continuous functions and $F_\sigma$-measurability}%
\thanks{This research was supported by the University of Silesia Mathematics Department (Iterative Functional Equations and Real Analysis program)}

\author[T. Banakh]{Taras Banakh}%
\address{Institute of Mathematics, Jan Kochanowski University in Kielce, \.Zeromskiego Str.  5, 25-369 Kielce, Poland and Department of Mechanics and Mathematics, Ivan Franko Lviv National University, Universytetska Str.  1, 79000 Lviv, Ukraine
}
\email{t.o.banakh@gmail.com}%

\author[O.  Maslyuchenko]{Oleksandr Maslyuchenko}%

\address{Institute of Mathematics, University of Silesia in Katowice, Bankowa 12, 40-007 Katowice, Poland and Yuriy Fedkovych Chernivtsi National University, Department of Mathematical Analysis, Kotsiubynskoho 2, 58012 Chernivtsi, Ukraine}
\email{ovmasl@gmail.com}%

%(15) (PDF) Extension of functions and metrics with variable domains. Available from: https://www.researchgate.net/publication/305851210_Extension_of_functions_and_metrics_with_variable_domains [accessed Feb 01 2019].
%\thanks{}%

%\date{3.03.2016}%
%\dedicatory{}%
%\commby{}%
% ----------------------------------------------------------------
\begin{abstract}
The linear continuity of a function defined on a vector space means that its restriction  to every affine line is continuous. For functions defined on $\mathbb R^m$ this notion is near to the separate continuity for which it is required  only the continuity on the straight lines which are parallel to coordinate axes.  The classical  Lebesgue theorem states that every separately continuous function $f:\mathbb R^m\to\mathbb R$ is of the $(m-1)$-th Baire class. In this paper we prove that every linearly continuous function $f:\mathbb R^m\to\mathbb R$ is of the first Baire class. Moreover, we obtain the following result. If $X$ is a Baire cosmic  topological vector space, $Y$ is a Tychonoff topological space and $f:X\to Y$ is a Borel-measurable (even BP-measurable)  linearly continuous function, then $f$ is $F_\sigma$-measurable. Using this theorem we characterize the discontinuity point set of an arbitrary linearly continuous function on $\mathbb R^m$. In the final part of the article we prove that  any $F_\sigma$-measurable function $f:\partial U\to \IR$ defined on the boundary of a strictly convex open set $U\subset\IR^m$ can be extended to a  linearly continuous function $\bar f:X\to \IR$. This fact shows  that  in the ``descriptive sense'' the linear continuity is  not better  than the $F_\sigma$-measurability.
\end{abstract}
\maketitle
% ----------------------------------------------------------------

%\def\<{\subset}           \def\>{\supset}           \def\*{\times}
%\def\ov{\overline}          \def\Int{{\rm int}}       \def\Cl{{\rm cl}}
%\def\pr{{\rm pr}}         \def\copr{{\rm copr}}     \def\dim{{\rm dim}}
%\def\fr{{\rm fr}}         \def\ext{{\rm ext}} \def\supp{{\rm supp}\,}
%\def\os{\stackrel{\underline{\phantom{\subset}}}{\subset}}
%\def\diam{{\rm diam}}     \def\span{{\rm span}}     \def\co{{\rm co}}
%\def\absco{{\rm absco}} \def\sup{\mathop{\rm sup}\limits}
%\def\inf{\mathop{\rm inf}\limits}  \def\lim{\mathop{\rm lim}\limits}
%\def\Lim{\mathop{\rm Lim}\limits}  \def\stm{\setminus}  \def\xp{{\rm xp}}
%\def\toto{\stackrel{\longrightarrow}{\scriptstyle\longrightarrow}}
%\def\lra{\Leftrightarrow}    \def\la{\Leftarrow}  \def\ra{\Rightarrow}
%\def\wg{\wedge} \def\lil{\lim\limits} \def\lr{{\rm lr}} \def\diam{{\rm diam}\,}
%\def\Fr{{\rm fr}\,} \def\Cr{\mathrm{cr}}
%\def\supl{\sup\limits}
%\def\infl{\inf\limits}
%%%%%%%%%%%%%%%%%%%%%%%%%%%%%%%%%%%%% Other %%%%%%%%%%%%%%%%%%%%%%%%%%%%%%%%%%
%\def\dis{\displaystyle}\def\~{\tilde}\def\wt{\widetilde}
%\def\#{\prec}\def\ua{U^\alpha}\def\ub{U^\beta}\def\va{V^\alpha}
%\def\vb{V^\beta}\def\lt{\ell_\iy(T)}\def\bigtu{\bigtriangleup}
\def\ov{\overline}\def\wt{\widetilde} \def\iy{\infty}
%\def\0{\varnothing}\def\ts{\textstyle}\def\Cr{\mathrm{cr}\,}
%\def\CL{CL}\def\KL{KL}\def\LL{LL}
%\def\olim{\mathop{\overline{\mathrm{Lim}}}\limits}
%\def\iu{^{\vee}}
%\def\iul{^{\vee\hspace{-0,3em}\wedge}}
%\def\iulu{^{\vee\hspace{-0,3em}\wedge\hspace{-0,3em}\vee}}
%\def\il{^{\wedge}}
%\def\ilu{^{\wedge\hspace{-0,3em}\vee}}
%\def\ilul{^{\wedge\hspace{-0,3em}\vee\hspace{-0,3em}\wedge}}
%\def\ra{\Rightarrow}
%\def\intl{\int\limits} \def\tod{\stackrel{d}{\to}}
%\def\pperp{^{\perp\hspace{-0,4em}\perp}}

%%%%%%%%%%%%%%%%%%%%%%%%%%%%%%%%%%%%%%%%%%%%%%%%%%%%%%%%%%%%%%%%%%%%%%%%%%%%%%%%%%%%%
%%%%%%%%%%%%%%%%%%%%%%%%%%%%%%%%%%%%%%%%%%%%%%%%%%%%%%%%%%%%%%%%%%%%%%%%%%%%%%%%%%%%%
%%%%%%%%%%%%%%%%%%%%%%%%%%%%%%%%%%%%%%%%%%%%%%%%%%%%%%%%%%%%%%%%%%%%%%%%%%%%%%%%%%%%%
%%%%%%%%%%%%%%%%%%%%%%%%%%%%%%%%%%%%%%%%%%%%%%%%%%%%%%%%%%%%%%%%%%%%%%%%%%%%%%%%%%%%%
%%%%%%%%%%%%%%%%%%%%%%%%%%%%%%%%%%%%%%%%%%%%%%%%%%%%%%%%%%%%%%%%%%%%%%%%%%%%%%%%%%%%%

\section{Introduction}

Separately continuous functions have been intensively studied the last 120 years, starting with the seminal dissertation of R. Baire  \cite{Baire}. The separate continuity of a function  of many variables means  the continuity with respect to each variable. This is equivalent to the continuity of the restrictions of the function  onto every  affine line, parallel to a coordinate axis. Requiring the continuity of the restrictions of the function on every affine line, we obtain the definition of a  linearly continuous function.

More precisely, a function $f:X\to Y$ from a topological vector space $X$ to a topological space $Y$ is {\em linearly continuous} if for any $x,v\in X$ the function $\IR\to Y$, $t\mapsto f(x+vt)$, is continuous. All topological vector spaces appearing  in this paper are over the field $\IR$ of real numbers and are assumed to be Hausdorff.

In contrast to the extensive literature on separate continuity, the number of papers devoted to the linear continuity is relatively small.

Maybe for the first time, linearly continuous functions appeared in the paper  \cite{Genocchi_Peano} containing an example of a discontinuous linearly continuous function $f:\IR^2\to\IR$. This function is defined by the formula  $f(x,y)=\frac{2xy^2}{x^2+y^4}$ where $f(0,0)=0$. An example of a linearly continuous function which is discontinuous at points of some set of cardinality continuum was constructed in \cite{Young}.  Slobodnik in \cite{Slobodnik} proved that the set $D(f)$ of discontinuity points of a linearly continuous function $f:\IR^m\to\IR$ is a countable union of isometric copies of the graphs of Lipschitz functions $h:K \to\IR$ defined on compact nowhere dense subsets $K$ of $\IR^{m-1}$. On the other hand, by a result of Ciesielski and Glatzer\cite{Ciesielski_Glatzer}, a subset $E\subset \IR^m$ coincides with the set $D(f)$ of discontinuity points of some  linearly continuous function $f:\IR^m\to\IR$ if $E$ is the countable union of closed nowhere dense subsets of convex surfaces. A similar result was obtained earlier in the paper  \cite{2008_1}, containing also a characterization of the sets $D(f)$ of discontinuity points of linearly continuous functions $f:\IR^n\to\IR$ of the first Baire class in terms of $\bar\sigma$-$\ell$-miserability. In this paper we shall generalize this characterization to linearly continuous BP-measurable functions defined on any Baire cosmic vector space.

First we prove that any real-valued linearly continuous function on a finite-dimensional topological vector space is of the first Baire class.

A function $f:X\to Y$ between topological spaces is defined to be
\begin{itemize}
\item {\em of the first Baire class} if $f$ is a pointwise limit of a sequence of  continuous functions from $X$ to $Y$;
\item {\em of $n$-th Baire class} for $n\ge 2$ if $f$ is a pointwise limit of a sequence of functions of the $(n-1)$-th Baire class from $X$ to $Y$.
\end{itemize}

It is well-known \cite{Rudin, 2001_2, 2004_1} that for every $n\ge 2$, each separately continuous  function $f:\IR^n\to \IR$ is of the $(n-1)$-th Baire class, and  $(n-1)$ in this result cannot be replaced by a smaller number. This fact contrasts with the following surprising property of linearly continuous functions, which will be proved in Section~\ref{s:pf:t1}.

\begin{thm}\label{t1} Every linearly continuous function $f:X\to \IR$ on a finite-dimensional topological vector space $X$ is of the first Baire class.
\end{thm}

Taking any discontinuous linear functional on an infinite-dimensional Banach space, we see that Theorem~\ref{t1} does not generalize to infinite-dimensional topological vector spaces. However, it is still true for BP-measurable linearly continuous functions between Baire cosmic vector spaces.

By a {\em cosmic vector space} we understand a topological vector space, which is a continuous image of a separable metrizable space. By \cite[Theorem 1]{BH} (or \cite[Lemma 5.1]{BR19}), every Baire cosmic topological group is separable and metrizable. This implies that every Baire cosmic vector space is separable and metrizable.

%A topological space is called {\em perfectly normal} if $X$ is normal and each open subset of $X$ is of type $F_\sigma$. It is easy to see that a topological space $X$ is perfectly normal if and only if each open subset of $X$ is functionally open. A topological space $X$ is Tychonoff if and only if it is a $T_1$-space and has a base of the topology consisting of functionally open sets.

We recall \cite{Kechris} that a subset $A$ of a topological space $X$ has {\em the Baire property} in $X$ if there exists an open set $U\subset X$ such that the symmetric difference $U\triangle A$ is meager in $X$. It is well-known \cite{Kechris} that the family of sets with the Baire property in a topological space $X$ is a $\sigma$-algebra containing the $\sigma$-algebra of Borel subsets of $X$.

A subset $U$ of a topological space $X$ is called  {\em functionally open} if $U=f^{-1}(V)$ for some continuous function $f:X\to\IR$ and some open set $V\subseteq \IR $. Observe that an open subset of a normal space is functionally open if and only if it is of type $F_\sigma$. A subset $A$ of a topological space is {\em functionally closed} if its complement $X\setminus A$ is functionally open in $X$.% A topological space $X$ is called {\em perfectly normal} if it is normal and each open set in $X$ is of type $F_\sigma$. It is well-known that a $T_1$-topological space $X$ is perfectly normal if and only if every open subset of $X$ is functionally open.

A function $f:X\to Y$ between topological spaces is called
\begin{itemize}
\item {\em $F_\sigma$-measurable} if for any functionally open set $U\subset Y$ the preimage $f^{-1}(U)$ is of type $F_\sigma$ in $X$.
\item {\em Borel-measurable} if for any functionally open set $U\subset Y$ the preimage $f^{-1}(U)$ is a Borel subset of $X$;
\item {\em BP-measurable} if for any functionally open set $U\subset Y$ the preimage $f^{-1}(U)$ has the Baire property in $X$.
\end{itemize}
It follows that each $F_\sigma$-measurable function is Borel-measurable and each Borel-measurable function is BP-measurable. By Theorem 1 of \cite{Fosgerau}, a function $f:X\to Y$ from a metrizable space $X$ to a connected locally path-connected separable metrizable space $Y$ is $F_\sigma$-measurable if and only if $f$ is of the first Baire class.

Now we can formulate one of the principal results of this paper.

\begin{thm}\label{t2} Every BP-measurable linearly continuous function $f:X\to Y$ from a Baire cosmic vector space $X$ to a Tychonoff space $Y$ is $F_\sigma$-measurable. If\/ $Y$ is a  separable, metrizable, connected and locally path-connected, then  $f$ is of the first Baire class.
\end{thm}

The proof of Theorem~\ref{t2} consists of two steps: first we establish that every BP-measurable linearly continuous function  on a Baire topological vector space is conically quasi-continuous, and then prove that every  conically quasi-continuous function  on a second-countable topological vector space is $F_\sigma$-measurable.

The conical quasi-continuity is defined for functions on topological vector spaces and is a modification of the quasi-continuity that takes into account the linear structure of the domain of the function.

Let us recall \cite{Neubrunn} that a function $f:X\to Y$ between topological spaces is {\em quasi-continuous} if for every point $x\in X$, neighborhood ${V}\subset X$ of $x$ and neighborhood ${W}\subset Y$ of $f(x)$, there exists a nonempty open set $U\subset {V}$ such that $f(U)\subset {W}$. Observe that a function $f:X\to Y$ is quasi-continuous if and only if for any open set $U\subset Y$ the preimage $f^{-1}(U)$ is quasi-open in the sense that the interior $f^{-1}(U)^\circ$ of $f^{-1}(U)$ is dense in $f^{-1}(U)$. This implies that every quasi-continuous function is BP-measurable.

A  subset $U$ of a topological vector space $X$ is called {\em conical at a point} $x\in X$ (or else {\em $x$-conical\/}) if $U\ne\emptyset$  and for every $u\in U$ the open segment $(x;u):=\big\{(1-t)x+tu:0<t<1\big\}$ is contained in $U$. It follows that each $x$-conical set contains $x$ in its closure.

A function $f:X\to Y$ from a topological vector space $X$ to a topological space $Y$ is called {\em conically quasi-continuous} if for any point $x\in X$, $x$-conical open set $V\subset X$ and open neighborhood ${W}\subset Y$ of $f(x)$, there exists an $x$-conical open set $U\subset V$ such that $f(U)\subset {W}$. It is easy to see that every conically quasi-continuous function on a topological vector space is quasi-continuous.  The converse is true for linearly continuous functions  on Baire topological vector spaces.

\begin{thm}\label{t3} For a linearly continuous function $f:X\to Y$ from a Baire topological vector space $X$ to a Tychonoff space $Y$ the following conditions are equivalent:
\begin{enumerate}
\item $f$ is conically quasi-continuous;
\item $f$ is quasi-continuous;
\item $f$ is BP-measurable.
\end{enumerate}
\end{thm}

\begin{thm}\label{t:cqc=>F} Each conically quasi-continuous function $f:X\to Y$ from a second-countable topological vector space $X$ to a topological space $Y$ is  $F_\sigma$-measurable.
\end{thm}

Theorems~\ref{t2}, \ref{t3} and \ref{t:cqc=>F} will be proved in Sections~\ref{s:t2}, \ref{s:BP} and \ref{s:tF}, respectively.

Now we discuss a characterization of sets $D(f)$ of discontinuity points of linearly continuous functions $f:X\to Y$ on Baire cosmic vector spaces, extending the characterization given by  the second author in  \cite{2008_1}.

A subset $V$ of a topological vector space $X$ is called an {\em $\ell$-neighborhood} of a set $A\subset X$ if for any $a\in A$ and $v\in X$ there exists $\e>0$ such that $a+[0;\e)\cdot v\subset V$.

Following \cite{2008_1}, we define a subset $A$ of a topological vector space $X$ to be
\begin{itemize}
\item {\em $\ell$-miserable} if $A\subset \overline{X\setminus L}$ for some closed $\ell$-neighborhood $L$ of $A$ in $X$;
\item {\em $\bar\sigma$-$\ell$-miserable}  if $A$ is a countable union of closed $\ell$-miserable sets in $X$.
%\item {\em $c$-nowhere dense} if there is an open convex set $U$ such that $A$ is nowhere dense in $\partial U$;
%\item {\em $\bar\sigma$-$c$-nowhere dense}  if $A$ is a countable union of closed $c$-nowhere dense sets in $X$;
\end{itemize}
The definition implies that each  $\bar\sigma$-$\ell$-miserable set in a topological vector space is of type $F_\sigma$.

There are many results describing the sets of discontinuity points of functions from various function classes  (see for example \cite{BreckenridgeNishiura, Grande,  Kershner, 2007_4,  2009_2, Maslyuchenko_Mykhaylyuk_1992, 2001_2})
In Section~\ref{s:miser} we shall prove the following description of the sets of discontinuity points of linearly continuous functions.

\bigskip
\begin{thm}\label{t:miser}\hspace{4cm}
\begin{enumerate}
\item For any BP-measurable linearly continuous function $f:X\to Y$ from a Baire cosmic vector space $X$ to a separable metrizable space $Y$, the set $D(f)$ of discontinuity points of $f$ is $\bar\sigma$-$\ell$-miserable in $X$.
\item For any $\bar\sigma$-$\ell$-miserable set $M$ in a metrizable topological vector space $X$, there exists a lower semicontinuous (and hence $F_\sigma$-measurable) linearly continuous function ${f:X\to[0;1]}$ such that $D(f)=M$.
\end{enumerate}
\end{thm}

A function $f:X\to\IR$ is called {\em lower semi-continuous} if for any $a\in\IR$ the set $\{x\in X:f(x)>a\}$ is open in $X$.

Many examples of $\ell$-miserable and $\bar\sigma$-$\ell$-miserable sets can be constructed using the following theorem, proved in Section~\ref{s:tCovMis}.

\begin{thm}\label{t:CovMis}  Any nowhere dense subset of the boundary $\partial U$ of an open convex set $U$ in a normed space $X$ is $\ell$-miserable in $X$.
\end{thm}

For finite-dimensional normed spaces, Theorem~\ref{t:CovMis} was proved in  \cite{2008_1} (see also Theorem 3.4 in \cite{Ciesielski_Glatzer}).

%Note that we do not know the answer to the following question: Is there exist $\bar\sigma$-$\ell$-miserable set in $\IR^2$ which is not $\bar\sigma$-$c$-nowhere dense?

Finally, we discuss one extension result showing that from the view-point of Descriptive Set Theory,  BP-measurable linearly continuous maps are not better than functions of the first Baire class.

A subset $C$ of a vector space $X$ is called {\em strictly convex} if for any distinct points $x,y\in \overline{C}$ the open segment $(x;y)$ is contained in $C$.

\begin{thm}\label{t:1.6} Let $K$ be a $\sigma$-compact functionally closed subset of a topological vector space $X$. If $K$ is contained in the boundary $\partial U$ of some strictly convex open set $U\subset X$, then every $F_\sigma$-measurable function $f:K\to Y$ to a  Banach space $Y$ can be extended to an $F_\sigma$-measurable linearly continuous function $\bar f:X\to Y$.
\end{thm}

Theorem~\ref{t:1.6} will be proved in Section~\ref{s:pf:t:1.6}. It has the following ``finite-dimensional'' corollary.

\begin{cor}\label{corExtencionOfFsigmaMaesurableFunctionInRn}
  Let $U$ be a strictly convex open set  in a finite-dimensional topological vector space $X$ and  $Y$ be a  Banach space.
  Every $F_\sigma$-measurable function $f:\partial U\to Y$ can be extended to a linearly continuous function $\bar f:X\to Y$.
\end{cor}

\section{Proof of Theorem~\ref{t1}}\label{s:pf:t1}

Given a linearly continuous function $f:X\to\IR$ on a finite-dimensional topological vector space $X$, we need to prove that $f$ is of the first Baire class.

Since $\IR$ is homeomorphic to $(0;1)$, we may assume  that $f(X)\subset (0;1)$. Also we can identify the finite-dimensional  topological vector space $X$ with the Euclidean space $\IR^m$ for some $m\ge 0$. Being linearly continuous, the function $f:\IR^m\to (0;1)$ is separately continuous and by \cite{Rudin}, $f$ is of the $(m-1)$-th Baire class and thus is Borel-measurable.

  Let
$\mu$ be the Lebegue measure on the Euclidean space $X=\IR^m$ and $B=\{x\in X:\|x\|\le 1\}$ be the closed unit ball in $X$. For every $n\in\IN$, consider the function
$$f_n:X\to(0;1),\;\;f_n:x\mapsto\tfrac{1}{\mu(B)}\int\limits_{B}f\big(x+{\tfrac{u}{n}}\big)d\mu(u),$$
which is continuous by Lemma 473(b) in \cite{Fremlin}.  The linear continuity of $f$ ensures that $\lim_{n\to\infty}f(x+\frac{u}n)=f(x)$ for any $x\in X$ and $u\in B$.
By  Lebesgue's dominated convergence theorem, $f_n(x)\to f(x)$ for every $x\in X$, which means that $f$ is of the first Baire class.

\section{Proof of Theorem~\ref{t3}}\label{s:BP}

In this section we prove Theorem~\ref{t3}. The implications $(1)\Ra(2)\Ra(3)\Ra(1)$ of this theorem follow from Lemmas~\ref{l:cqc=>qc}, \ref{l:qc=>BP}, \ref{l:BP-cqc}, respectively.

\begin{lem}\label{l:cn} For any open set $U$ in a topological vector space $X$ and any point $x\in U$ there exists an open $x$-conical neighborhood $V\subset U$ of $x$.
\end{lem}

\begin{proof} By the continuity of the map $\gamma:X\times[0,1]\to X$, $\gamma:(z,t)\mapsto tz+(1-t)x$, (with $\gamma\big(\{x\}\times[0;1]\big)=\{x\}$) and the compactness of the closed interval $[0;1]$, there exists an open neighborhood $W\subset X$ of $x$ such that $\gamma\big(W\times[0;1]\big)\subset U$. It follows that $$V:=\gamma\big(W\times (0;1]\big)=\bigcup_{0<t\le 1}\big((1-t)x+tW\big)$$ is a required open $x$-conical neighborhood  in $U$.
\end{proof}

 \begin{lem}\label{l:cqc=>qc} Every conically quasi-continuous function $f:X\to Y$ from a topological vector space $X$ to a topological space $Y$ is quasi-continuous.
\end{lem}

\begin{proof} Given a point $x\in X$ and two open sets ${O}\subset X$ and ${W}\subset Y$ with $x\in {O}$ and $f(x)\in {W}$, we need to find a non-empty open set $U\subset {O}$ such that $f(U)\subset {W}$. By Lemma~\ref{l:cn}, the neighborhood ${O}$ of $x$ contains an open $x$-conical neighborhood $V$ of $x$. By the conical quasi-continuity of $f$, there exists a open $x$-conical set $U\subseteq V\subseteq {O}$ such that $f(U)\subset {W}$. By definition of an $x$-conical set, $U$ is not empty, witnessing that $f$ is quasi-continuous.
\end{proof}

\begin{lem}\label{l:qc=>BP} Every quasi-continuous function $f:X\to Y$ between topological spaces is BP-measurable.
\end{lem}

\begin{proof} To prove that $f$ is BP-measurable, fix any open set $U\subset Y$. The quasi-continuity of $f$ ensures that the interior $f^{-1}(U)^\circ$ of the preimage $f^{-1}(U)$ is dense in $f^{-1}(U)$. Then the set $f^{-1}(U)\setminus f^{-1}(U)^\circ$ is nowhere dense and hence $f^{-1}(U)$ has the Baire property in $X$, witnessing that the function $f$ is BP-measurable.
\end{proof}

\begin{lem}\label{l:BP=>G}
Let $f:X\to Y$ be a BP-measurable linearly continuous function from a Baire topological vector space $X$ to a topological space $Y$. For any point $x\in X$, functionally open neighborhood ${W}\subseteq Y$ of $f(x)$ and $x$-conical open set $V\subseteq X$ there exist  an $x$-conical open set $U\subseteq V$ and a dense $G_\delta$-subset $G$ of $U$ such that $f(G)\subseteq {W}$.
\end{lem}

\begin{proof}  The BP-measurability of $f$ ensures that the preimage $f^{-1}({W})$ has the Baire property in $X$.

For every $n\in\IN$, consider the $x$-conical subset $V_n=\big\{v\in V:x+[0;\tfrac1n]\cdot v\subset f^{-1}({W})\big\}$ of $V$ and observe that $V=\bigcup_{n\in\IN}V_n$ by the linear continuity of $f$. Since the open set $V$ is not meager in the Baire space $X$, for some $n\in\IN$ the set $V_n$ is not meager in $X$. Consequently, there exists a non-empty open set $U'\subset V$ such that $U'\cap V_n$ is a dense Baire subspace of $U'$. Consider the open $x$-conical subset $U=\big\{x+tu:u\in U',\;0<t\le\frac1n\big\}$ of  $V$ and observe that $U\cap (x+\frac1nV_n)$ is a dense Baire subspace of $U$. This follows from the observation that for any point $u\in U$ there exists  $t\in(0;\frac1n]$ and $u'\in U'$ such that $u=x+tu'$ and then $x+tU'$ is an open neighborhood of $u$ in $U$ such that $(x+tU')\cap (x+\frac1nV_n)=(x+tU')\cap(x+[0,\tfrac1n]V_n)\supset x+t(U'\cap V_n)$ is a dense Baire subspace of $x+tU'$.

Now observe that the intersection $U\cap  f^{-1}({W})$ has the Baire property in $X$ and contains the dense Baire subspace $U\cap (x+\tfrac1nV_n)$ of $U$, which implies that $U\cap f^{-1}({W})$ contains a dense $G_\delta$-subset $G$ of $U$.
\end{proof}

\begin{lem}\label{l:BP-cqc} Every BP-measurable linearly continuous function $f:X\to Y$ from a Baire topological vector space $X$ to a Tychonoff space $Y$ is conically quasi-continuous.
\end{lem}

\begin{proof} Given a point $x\in X$, an open $x$-conical set $V\subset X$ and a neighborhood ${O}\subset Y$ of $f(x)$, we need to find an $x$-conical open set $U\subset V$ such that $f(U)\subset {O}$. Since the space $Y$ is Tychonoff, there exists a functionally open neighborhood $W$ of $f(x)$ such that $\overline{W}\subset {O}$.

By Lemma~\ref{l:BP=>G}, there exists an open $x$-conical set $U\subset V$ and a dense $G_\delta$-set $G$ in $U$ such that $f(G)\subset W$. We claim that $f(U)\subset\overline{W}\subset {O}$. To derive a contradiction, assume that $f(u)\notin\overline{W}$ for some $u\in U$. Since $Y$ is Tychonoff, the point $f(u)$ has a  functionally open neighborhood $W_u\subset X\setminus\overline{W}$. By Lemma~\ref{l:cn}, there exists an open $u$-conical neighborhood $V_u\subset U$ of $u$. By Lemma~\ref{l:BP=>G}, there exists an open $u$-conical set $U_u\subset V_u$ containing a dense $G_\delta$-subset $G_u$ such that $f(G_u)\subset W_u$. Then $G_u$ and $U_u\cap G$ are two dense $G_\delta$-subsets of the space $U_u$. Since $U_u$ is Baire, the intersection $G_u\cap(U_u\cap G)=G_u\cap G$ is not empty. On the other hand,
$$f(G_u\cap U_u\cap G)=f(G_u\cap G)\subset f(G_u)\cap f(G)\subset W_u\cap W\subset (Y\setminus \overline{W})\cap W=\emptyset,$$
and this is a desired contradiction completing the proof.
\end{proof}

\section{Proof of Theorem~\ref{t:cqc=>F}}\label{s:tF}

Given a conically quasi-continuous function $f:X\to Y$ from a second-countable topological vector space $X$ to a topological space $Y$, we need to prove that $f$ is $F_\sigma$-measurable. To derive a contradiction, assume that the function $f$ is not $F_\sigma$-measurable. Then there exists a functionally open subset $G\subset Y$ such that $A=f^{-1}(G)$ is not of type $F_\sigma$ in  $X$.

We say that a subset $B\subset A$ can be separated from $X\setminus A$ by an $F_\sigma$-set if there exists an $F_\sigma$-set $F\subset X$ such that $B\subset F\subset A$. It follows that $A$ cannot be separated from $X\setminus A$ by an  $F_\sigma$-set. Moreover, for any countable cover $\mathcal C$ of $A$ there exists a set $C\in\mathcal C$ that cannot be separated from $X\setminus A$ by an $F_\sigma$-set.

Since $G$ is functionally open in $Y$, there exists a continuous function $\varphi:Y\to[0;1]$ such that $G=\varphi^{-1}\big((0;1]\big)$. For every $m\in\IN$ consider the open set $G_m=\varphi^{-1}\big((\tfrac1m;1]\big)$ and observe that $G=\bigcup_{m\in\IN}G_m=\bigcup_{m\in\IN}\overline{G_m}$.

Since $A=\bigcup_{m\in\IN}f^{-1}(G_m)$ cannot be separated from $X\setminus A$ by an $F_\sigma$-set, for some $m\in\IN$ the set $A_m=f^{-1}(G_m)$ cannot be separated from $X\setminus A$ by an $F_\sigma$-set.

Fix a countable base $\{B_n\}_{n\in\w}$ of the topology of the second-countable space $X$, consisting of non-empty open sets. For every $n\in\IN$ consider the open $0$-conical set $\check B_n=(0;1]\cdot B_n$. Observe that every open $0$-conical subset of $X$ contains some set $\check B_n$.
This fact and the conical continuity of $f$ imply that for every point $x\in A_m$ there exists a number $n_x\in\IN$ such that $f(x+\check B_{n_x})\subset G_m$.
For every $n\in\IN$ consider the subset $A_{m,n}:=\{x\in A_m:n_x=n\}$.
Since the set $A_m=\bigcup_{n\in\IN}A_{m,n}$ cannot be separated from $X\setminus A$ by an $F_\sigma$-set, for some $n\in\IN$ the set $A_{m,n}$ cannot be separated from $X\setminus A$ by an $F_\sigma$-set. Then the closure $\overline{A_{m,n}}$ of $A_{m,n}$ in $X$
has a common point $y$ with $X\setminus A$. It follows that $f(y)\notin G$ and hence $f(y)\notin\overline{G_m}$. By the conical quasi-continuity of $f$, there exists $k\in\IN$ such that $y+\check B_k\subset y+\check B_m$ and $f(y+\check B_k)\subset Y\setminus\overline{G_m}$. It follows that $\check B_k\subset\check B_n$ and hence $y+\check B_k-\check B_k\subset y+\check B_k-\check B_n$ is a neighborhood of $y$ in $X$. Since $y\in\overline{A_{m,n}}$, there exists a point $z\in (y+\check B_k-\check B_n)\cap A_{m,n}$. For this point $z$ the sets $z+\check B_n$ and $y+\check B_k$ have non-empty intersection. On the other hand,
$$f\big((z+\check B_n)\cap (y+\check B_k)\big)\subset f(z+\check B_n)\cap f(y+\check B_k)\subset G_m\cap\big(Y\setminus\overline{G_m}\,\big)=\emptyset,$$
which is a contradiction that completes the proof of Theorem~\ref{t:cqc=>F}.

\section{Proof of Theorem~\ref{t2}}\label{s:t2}

Let $f:X\to Y$ be a BP-measurable linearly continuous function from a Baire cosmic vector space $X$ to a Tychonoff space $Y$. By \cite[Theorem 1]{BH}, the space $X$ is separable and metrizable (being a Baire cosmic topological group). By Lemma~\ref{l:BP-cqc}, the function $f$ is conically quasi-continuous and by Theorem~\ref{t:cqc=>F}, $f$ is $F_\sigma$-measurable. If the space $Y$ is separable, metrizable, connected and locally path-connected, then $f$ is of the first Baire class according to the Fosgerau Theorem 1 in \cite{Fosgerau}.

\section{Proof of Theorem~\ref{t:miser}}\label{s:miser}

The two statements of Theorem~\ref{t:miser} are proved in Lemmas~\ref{l:miser2} and \ref{l:miser4}.

\begin{lem}\label{l:miser1}  For any $F_\sigma$-measurable quasi-continuous linearly continuous function ${f:X\to Y}$ from a  topological vector space $X$ to a metrizable separable space $Y$, the set $D(f)$ of discontinuity points of $f$ is $\bar\sigma$-$\ell$-miserable in $X$.
\end{lem}

\begin{proof} Let $\{V_n\}_{n\in\IN}$ be a countable base of the topology of the separable metrizable space $Y$. For every $n\in\IN$, consider the set
$E_n=f^{-1}(V_n)\setminus\overline{f^{-1}(V_n)}^\circ$.
Since  $f$ is $F_\sigma$-measurable, $E_n$ is $F_\sigma$-set in $X$. Then $E_n=\bigcup_{k\in\IN} E_{nk}$ for suitable closed sets $E_{nk}$ in $X$.

Let us prove that $D(f)\subseteq\bigcup_{n\in\IN}E_n$. Fix a point $x\in D(f)$. Then there exists a set $A\subseteq X$ such that $x\in \overline{A}$ and $f(x)\not\in \overline{f(A)}$. Pick $n\in\mathbb{N}$ such that $f(x)\in V_n$ and $\overline{V_n}\cap f(A)=\emptyset$.

Let us show that $x\in E_n$. Obviously, $x\in f^{-1}(V_n)$.
Assuming that $x\in\overline{f^{-1}(V_n)}^\circ$, we can find a point $a\in A\cap
\overline{f^{-1}(V_n)}^\circ$. Since $f(a)\notin\overline{V_n}$, we can use the quasi-continuity of $f$ and find a non-empty open set $U\subset \overline{f^{-1}(V_n)}^\circ$ such that $f(U)\cap\overline{V_n}=\emptyset$. On the other hand, since $U\subset \overline{f^{-1}(V_n)}$, there exists a point $u\in U\cap f^{-1}(V_n)$. Then $f(u)\in V_n\cap f(U)\subset V_n\cap (Y\setminus\overline{V_n})=\emptyset$,
which is a desired contradiction, completing the proof of the inclusion $D(f)\subset\bigcup_{n\in\IN}E_n$.

Now we prove $\bigcup_{n\in\IN}E_{n}\subseteq D(f)$. Fix $n\in\IN$ and  $x\in E_{n}$. Then
$x\in E_n=f^{-1}(V_n)\setminus \overline{f^{-1}(V_n)}^\circ$. Consider the set $A=X\setminus f^{-1}(V_n)$ and observe that $\overline{A}=\overline{X\setminus
f^{-1}(V_n)}=X\setminus  f^{-1}(V_n)^\circ\supset X\setminus\overline{f^{-1}(V_n)}^\circ\ni x$ and $\overline{f(A)}\subseteq Y\setminus V_n\not\ni f(x)$, which implies $x\in D(f)$.

It is remains to prove that the sets $E_{nk}$ are $\ell$-miserable. Let $L_n=\overline{f^{-1}(V_n)}$. It is clear that $L_n$ is a closed $\ell$-neighborhood of $E_{nk}$ (as $E_{nk}\subset E_n\subset f^{-1}(V_n)$). Furthermore, $E_{nk}\subseteq E_n\subseteq X\setminus L_n^\circ= \overline{X\setminus L_n}$.
\end{proof}

Theorems~\ref{t2}, \ref{t3} and Lemma~\ref{l:miser1} imply

\begin{lem}\label{l:miser2}  For any $BP$-measurable linearly continuous function $f:X\to Y$ from a Baire cosmic vector space $X$ to a metrizable separable space $Y$, the set $D(f)$ of discontinuity points of $f$ is $\bar\sigma$-$\ell$-miserable in $X$.
\end{lem}

\begin{lem}\label{lemLCFwithMiserableD(f)}
  Let $X$ be a metrizable topological vector space and $F$ be a closed $\ell$-miserable set in $X$. Then there exists a lower semicontinuous linearly continuous function  ${f:X\to [0;1]}$ such that $D(f)=F\subseteq f^{-1}(0)$.
\end{lem}

\begin{proof}
Since the set $F$ is $\ell$-miserable in $X$, there exists a closed  $\ell$-neighborhood   $L$ of $F$
such that $F\subseteq\overline{X\setminus L}$.  Applying Corollary 2.4 and Proposition 2.2 from
\cite{2009_2}, we can find a set  $A\subseteq X\setminus L$ such that $\overline{A}\cap L=F$.
Consider the subspace $Y=X\setminus F$ of $X$ and observe that $B=\overline{A}\cap Y$ and $C=L\cap Y$ are disjoint closed sets in $Y$. By the Urysohn lemma \cite[1.5.11]{Engelking}, there exists a continuous function ${g:Y\to [0;1]}$ such that $g(B)\subset\{1\}$ and $g(C)\subset\{0\}$. Define $f:X\to[0;1]$ by  $f(x)=g(x)$ for $x\in Y$ and $f(x)=0$ for $x\in F$. It follows that  $f(L)=f(C)\cup f(F)\subset\{0\}$. But $L$ is  $\ell$-neighborhood of $F$. Therefore,  $f$ is linearly continuous at every point $x\in F$. On the other hand, the continuity of $f{\restriction}Y$ on the open subset $Y$ of $X$ ensures that $D(f)\subseteq F$. Thus  $f$ is continuous at every point $x\not\in F$. Consequently,  $f$ is linearly continuous. Taking into account that $F\subset\overline{A}$ and $f(A)\subset\{1\}$, we conclude that $D(f)=F\subseteq f^{-1}(0)$, which implies that $f$ is lower semicontinuous.
\end{proof}

\begin{lem}\label{l:miser4}
  Let $X$ be a metrizable topological vector space and $E$ be a $\bar\sigma$-$\ell$-miserable set in $X$. Then there exists a  lower semicontinuous linearly continuous function  $f:X\to [0;1]$ with $D(f)=E$.
\end{lem}

\begin{proof}  Since $E$ is $\bar\sigma$-$\ell$-miserable, there exist closed $\ell$-miserable sets $F_n$ such that $E=\bigcup_{n=1}^\infty F_n$. By Lemma~\ref{lemLCFwithMiserableD(f)} there are lower semicontinuous  linearly continuous functions $f_n:X\to [0;1]$ with  $D(f_n)=F_n\subseteq f^{-1}(0)$.  Define $f=\sum_{n=1}^\infty\tfrac{1}{2^n}f_n.$
Obviously, this series is uniformly convergent. So, $f$ is lower continuous and linearly continuous, being a uniform limit of lower semicontinuous linearly continuous functions. Finally, \cite[Lemma 1]{Maslyuchenko_Mykhaylyuk_1992}  implies that $D(f)=\bigcup_{n=1}^\infty D(f_n)=\bigcup_{n=1}^\infty F_n=E.$
\end{proof}

\section{Proof of Theorem~\ref{t:CovMis}}\label{s:tCovMis}

Given an open convex set $U$ in a normed space $X$ and a nowhere dense subset $F\subset\partial U$, we have to prove that $F$ is $\ell$-miserable in $X$. This is trivially true if $F$ is empty. So, we assume that $F\ne\emptyset$. Then the boundary $S=\partial U$ of $U$ is not empty as well.

For a non-empty subset $E\subset X$ denote by $d_E:X\to\IR$ the continuous function assigning to every $x\in X$ the distance $d_E(x)=\inf_{y\in E}\|x-y\|$ to the subset $E$. Consider the sets
$$G=\big\{x\in U:d_S(x)<(d_F(x))^2\big\}\mbox{ \ and \ }L=X\setminus G.$$
Obviously, $G$ is open and  $L$ is closed. Since  $d_F(x)=d_S(x)=0$ for any $x\in F$, we have $F\subset L$. The nowhere density of $F$ in $S$ implies that $F\subset\overline{S\setminus F}\subset\overline{G}=\overline{X\setminus L}$.

It remains to prove that $L$ is an $\ell$-neighborhood of $F$ in $X$. Given any non-zero vector $v\in X$ and any $x\in F$, we should find $\e>0$ such that $x+[0;\e)\cdot v\subset L$. If $x+[0;+\infty)\cdot v\subset L$, then we are done. So, assume that $x+\lambda v\notin L$ for some $\lambda>0$. Replacing $v$ by $\lambda v$ we can assume that $\lambda=1$ and hence $x+v\notin L$ and $x+v\in G\subset U$. Since $U$ is open, there exists $\delta>0$ such that $x+v+\delta B\subset U$, where $B=\{x\in X:\|x\|<1\}$ denotes the open unit ball in $X$. We claim that $x+[0;\e)\cdot v\in L$ where $\e:=\frac{\delta}{\|v\|^2}$. Indeed, for any positive $t<\e$, by the convexity of the set $\overline{U}$ we get $(1-t)x+tU\subset\overline{U}$. Being open, the set $(1-t)x+tU$ is contained in the interior $\overline{U}\setminus\partial U=U$ of $\overline{U}$. Then $$x+tv+t\delta B=(1-t)x+t(x+v+\delta B)\subset (1-t)x+tU\subset U$$and hence
$$d_S(x+tv)\ge t\delta={\|tv\|^2}\frac{\delta}{t\|v\|^2}>d_F(x+tv)^2\frac{\delta}{\e\|v\|^2}=d_F(x+tv)^2,$$which implies that $x+tv\notin G$ and hence $x+tv\in L$.

%It is remains  to prove  $F\subset \overline{G}$. Fix  $x\in F$ and consider an open neighborhood $W$ of $x$. Choose $y\in S$  and  $\e\in(0;1)$ such that $B(y,2\e)\subset W\setminus F$. Therefore, $d_F(y)>2\e$. Pick $z\in U\cap B(y,\e^2)$. Hence, $d_S(z)<\|z-y\|<\e^2$ and $2\e<d_F(y)\le\|y-z\|+d_F(z)<d_F(z)+\e^2<d_F(z)+\e$. So,  $d_F(z)>\e$. Therefore, $\big(d_F(z)\big)^2>\e^2>d_S(z)$ and then $z\in G$. But $z\in B(y,\e^2)\subset B(y,2\e)\subset W$. Hence $W\cap G\ne\emptyset$. Thus, $x\in \overline{G}$

\section{{Proof of Theorem~\ref{t:1.6}}}\label{s:pf:t:1.6}

The proof of Theorem~\ref{t:1.6} is preceded by four lemmas.

\begin{lem}\label{l:subnorm} Let $U$ be a strictly convex open set in a topological vector space $X$. If $\partial U\ne\emptyset$, then $X$ admits a continuous norm $\|\cdot\|$ such that the set $U$ remains open and strictly convex in the normed space $(X,\|\cdot\|)$.
\end{lem}

\begin{proof} If $\partial U$ is not empty, then the open convex set $U$ is not empty and contains some point. Replacing $U$ by a suitable shift, we can assume that this point is zero. Consider the open convex symmetric subset $B=U\cap(-U)$ and its gauge functional
$\|\cdot\|_B:X\to[0,\infty)$, assigning to each point $x\in X$ the number $\|x\|_B=\inf\{t>0:x\in tB\}$. By \cite[1.4]{Sch}, $\|\cdot\|_B$ is a seminorm on $X$. To show that $\|\cdot\|_B$ is a norm, it  suffices to check that $\|x\|_B>0$ for any non-zero element $x\in X$. Assuming that $\|x\|_B=0$, we conclude that $\IR\cdot x\subset B\subset U$. By our assumption, the boundary $\partial U$ contains some point $b$. Proposition 1.1 of \cite{BP} implies that $b+\IR\cdot x\subset\partial U$, which contradicts the strict convexity of $U$.  So, $\|\cdot\|$ is a continuous norm on $X$.

To see that the set $U$ remains open in the normed space $(X,\|\cdot\|)$, take any point $x\in U$. Since $U$ is open, there exists $\e\in(0;\frac12]$ such that $\frac1{1-\e}x\in U$. Since $U$ is convex, $x+\e B\subset x+\e U=(1-\e)\tfrac1{1-\e}x+\e U\subset U,$ which means that $x$ is an interior point of $U$ in the normed space $(X,\|\cdot\|)$.

Next, we show that each point $y\in X\setminus \overline{U}$ does not belong to the closure of $U$ in the normed space $(X,\|\cdot\|)$. Since $y\notin\overline{U}$, there exists $\delta\in(0,1]$ such that $\frac1{1+\delta}y\notin\overline{U}$. We claim that $y+\delta B$ is disjoint with $U$. In the opposite case we can find a point $u\in U\cap(y+\delta B)$. Then $y\in U+\delta B\subset U+\delta U=(1+\delta)U$ and  hence $\frac1{1+\delta}y\in U$, which is a desired contradiction.

Therefore the closure of $U$ in $X$ coincides with the closure of $U$ in the norm $\|\cdot\|$. Consequently, the set $U$ remains strictly convex in the norm $\|\cdot\|$.
\end{proof}

\begin{lem}\label{lemMIsLnbhd}
  Let  $U$ be a strictly convex open set  in a normed space $X$, and $K$ be a compact subset of the boundary $S=\partial U$ of $U$ in $X$. Then there exists a closed $\ell$-neighborhood $L$ of $K$ such that $L\cap S=K$.
\end{lem}

\begin{proof} Let $B=\{x\in X:\|x\|<1\}$ be the open unit ball of the normed space $X$. If $K$ is empty, then we put $L=\emptyset$ and finish the proof. Now assume that $K\ne\emptyset$. In this case the sets $S=\partial U$ and $X\setminus U$ are not empty.

For every point $x\in S\setminus K$, consider the compact set $\frac12x+\frac12K:=\{\frac12x+\frac12y:y\in K\}$, which is contained in $U$ by the strict convexity of $U$. By the compactness of $\frac12x+\frac12K$, the number
$$\delta(x)=\inf\big\{\|y-z\|:y\in \tfrac12x+\tfrac12K,\;\;z\in X\setminus U\big\}$$ is strictly positive. For every $y\in K$ we get
$$
\|x-y\|=2\|x-(\tfrac12x+\tfrac12y)\|\ge 2\delta(x),\eqno(*)
$$
which implies that the open ball $x+2\delta(x)B$  does not intersect $K$.

Let $\epsilon(x)=\min\{\delta(x),\delta(x)^2\}$ for $x\in S\setminus K$.
It  follows that the open set $$W=\bigcup_{x\in S\setminus K}(x+\epsilon(x)B)$$ contains $S\setminus K$ and is disjoint with $K$. Consequently, the closed set $L=X\setminus W$ has intersection $L\cap S=K$. It remains to prove that $L$ is an $\ell$-neighborhood of $K$. Given any $x\in K$ and any non-zero vector $v\in X$, we should find $\e>0$ such that $x+[0;\e)\cdot v\subset L$. If $x+[0;\infty)\cdot v\subset L$, then there is nothing to prove. So, assume that $x+rv\notin L$ for some $r>0$. Replacing the vector $v$ by $rv$, we can assume that $r=1$ and hence $x+v\notin L$. Then $x+v\in s+\epsilon(s)B$ for some $s\in S\setminus K$.
We claim that for every positive $t<\min\{\frac12,\frac{\epsilon(s)}{\|v\|^2}\}$ we have $x+tv\in L$. To derive a contradiction, assume that $x+tv\notin L$ and hence $x+tv\in s'+\epsilon(s')B$ for some $s'\in S\setminus K$.
The inequality $(*)$ implies that $2\delta(s')\le \|s'-x\|<t\|v\|+\epsilon(s')\le t\|v\|+\delta(s')$ and hence $\delta(s')<t\|v\|$. Then $\epsilon(s')\le\delta(s')^2< t^2\|v\|^2$ and
$$\tfrac12\epsilon(s)+\tfrac{\epsilon(s')}{2t}< \tfrac12\epsilon(s)+\tfrac{t^2\|v\|^2}{2t}=\tfrac12\epsilon(s)+\tfrac12\|v\|^2t<\tfrac12\epsilon(s)+\tfrac12\|v\|^2\tfrac{\epsilon(s)}{\|v\|^2}=\epsilon(s)\le \delta(s).$$
Finally $$
\begin{aligned}
s'&\in x+t v-\epsilon(s')B\subset x+t(s-x+\epsilon(s)B)-\epsilon(s')B=
(1-t)x+t s+(t\epsilon(s)+\epsilon(s'))B=\\
&=(1-2t)x+2t\big(\tfrac12x+\tfrac12s+(\tfrac12\epsilon(s)+\tfrac{\epsilon(s')}{2t})B\big)\subset (1-2t)x+2t\big(\tfrac12s+\tfrac12x+\delta(s)B\big)\subset\\
&\subset (1-2t)x+2tU\subset U,
\end{aligned}
$$
which is not possible as $s'\in S\subset X\setminus U$.
\end{proof}

  A function $f:X\to Y$ between topological spaces $X$ and $Y$ is called \emph{$\bar\sigma$-continuous} if there exists a sequence of closed sets $F_n$ such that $X=\bigcup_{n\in\w}F_n$ and the restriction $f{\restriction}{F_n}$ is continuous for any $n\in\w$.

\begin{lem}\label{lemFsigmaMeasurableFunctionAsSumOfSigmaContinuous}
  Let $X$ be a perfectly normal space, $Y$ be a separable normed space and ${f:X\to Y}$ be an $F_\sigma$-measurable function. Then there is a sequence of $\bar\sigma$-continuous functions $f_n:X\to Y$ such that $f(x)=\sum_{n=0}^\infty f_n(x)$ and  $\sup_{x\in X}\|f_{n}(x)\|\le\frac1{2^n}$ for every $n\in\IN$.
\end{lem}

\begin{proof}
Fix any countable dense set  $\{y_k:k\in\w\}$ in the separable normed space $Y$. Let $B=\{y\in Y:\|y\|<1\}$ be the open unit ball in $Y$. By the $F_\sigma$-measurability of $f$, for every $n,k\in\w$, the set $A_{n,k}:=f^{-1}(y_k+\frac1{2^{n+2}}B)$ is of type $F_\sigma$ in $X$. Obviously, $X=\bigcup_{k\in\w} A_{n,k}$ for every $n\in\w$. By the reduction theorem \cite[p.358]{Kuratowski1}, for every $n\in\w$ there exists a disjoint sequence $(E_{n,k})_{k\in\w}$ of $F_\sigma$-sets $E_{n,k}\subseteq A_{n,k}$ such that $\bigcup_{k\in\w}E_{n,k}=X$. Write each $F_\sigma$-set $E_{n,k}$ as the countable union $E_{n,k}=\bigcup_{j\in\w}F_{n,k,j}$ of closed sets $F_{n,k,j}$. For every $n\in\w$, consider the function $g_n:X\to Y$ assigning to each $x\in X$ the point $y_k$ where $k\in\w$ is the unique number such that $x\in E_{n,k}\subset A_{n,k}$. Then $\|f(x)-g_n(x)\|=\|f(x)-y_k\|<\frac1{2^{n+2}}$. Since $g_n{\restriction}{F_{n,k,j}}$ is constant and $X=\bigcup_{k,j\in\w}F_{n,k,j}$, the function $g_n$ is $\bar\sigma$-continuous. Put $f_0=g_0$ and $f_{n}=g_{n}-g_{n-1}$ for $n\in\IN$. Then $$\|f_n(x)\|\le \|g_n(x)-f(x)\|+\|f(x)-g_{n-1}(x)\|<\tfrac1{2^{n+2}}+\tfrac1{2^{n+1}}<\tfrac1{2^n}$$for every $x\in X$ and
$\sum_{n=0}^\infty f_n=\lim_{n\to\infty}g_n=f$.
\end{proof}

A topological space $Y$ is called an {\em absolute extensor} if every continuous map $f:X\to Y$ defined on a closed subspace $X$ of a metrizable space $M$ has a continuous extension $\bar f:M\to Y$. By a classical Dugundji result \cite{Dugundji},  every convex subset of a normed space is an absolute extensor.

%\begin{lem}\label{l:AR} Let $Y$ be a Polish absolute retract. Any continuous map $f:A\to Y$ defined on a closed subspace $A$ of a normal topological space $X$ admits a continuous extension $\bar f:X\to Y$.
%\end{lem}

%\begin{proof} By the definition of a Polish absolute retract, we can identify $Y$ with a retract of the countable product of lines $\IR^\w$. Fix any retraction $r:\IR^\w\to Y$ of $\IR^\w$ onto $Y$, which is a continuous map such that $r(y)=y$ for all $y\in Y$, By the Tietze-Urysohn Theorem~\cite[2.18]{Engelking}, any map $f:A\to Y$ defined on a closed subset $A$ of a perfectly normal space $X$ has a continuous extension $\tilde f:X\to\IR^\w$. Then $\bar f=r\circ \tilde f:X\to Y$ is a required continuous extension of the map $f$.
%\end{proof}

\begin{lem}\label{lemExtentionOsSigmaContinuousFunction}
  Let $U$ be a strictly convex subset of a  topological vector space $X$,  $\partial U$ be the boundary of $U$ in $X$ and $K\subset\partial U$ be a  $\sigma$-compact functionally closed set in $X$. Any $\bar\sigma$-continuous function $f:K\to Y$ to an absolute extensor $Y$ can be extended to a  linearly continuous $\bar\sigma$-continuous function $\bar f:X\to Y$.
\end{lem}

\begin{proof} There is nothing to prove that $K=\emptyset$. So, we assume that $K\ne\emptyset$. In this case $\partial U\ne\emptyset$ and by Lemma~\ref{l:subnorm} the space $X$ admits a norm $\|\cdot\|$ such that the set $U$ remains open and strictly convex in the normed space $(X,\|\cdot\|)$.

Fix any $\bar\sigma$-continuous function $f:K\to Y$ to an absolute extensor $Y$. It follows that the $\sigma$-compact set $K$ can be written as the countable union  $K=\bigcup_{n\in\w}K_n$ of an increasing sequence $(K_n)_{n\in\w}$ of compact sets such that for every $n\in\w$ the restriction $f{\restriction}{K_n}$ is continuous.

Since $K$ is functionally closed in $X$, there exists a continuous function $\varphi:X\to[0,1]$ such that $K=\varphi^{-1}(0)$. On the space $X$ consider the continuous metric $\rho$ defined by $\rho(x,y)=\|x-y\|+|\varphi(x)-\varphi(y)|$. For a point $x\in X$ and a non-empty subset $A\subset X$ let $\rho(x,A)=\inf_{a\in A}\rho(x,a)$.

 For every $n\in\w$ consider the $\rho$-open set $G_n=\{x\in X:\rho(x,K)>\frac1{2^n}\}$ and the $\rho$-closed set $\tilde G_n=\{x\in X:\rho(x,K)\ge\frac1{2^n}\}$ in $X$ and observe that $\bigcup_{n\in\w}G_n=\bigcup_{n\in\w}\tilde G_n=X\setminus K$.

By Lemma~\ref{lemMIsLnbhd}, for every $n\in\w$ there is a closed $\ell$-neighborhood $L_n' $ of $K_n$ in the normed space $(X,\|\cdot\|)$ such that $L_n'\cap K=K_n$. For every $n\in\w$ consider the $\rho$-closed set $L_n=\tilde{G}_n\cup\bigcup_{k\le n}L_k'$ and observe that $L_n$ is an $\ell$-neighborhood of $K_n$ such that $L_n\cap K=K_n$. It follows that $\bigcup_{n\in\w}L_n=X$.

A function $g:A\to Y$ on a subset $A\subset X$ will be called {\em $\rho$-continuous} if it is continuous with respect to the topology on $A$, generated by the metric $\rho$. Observe that for every $n\in\w$ the restriction $f{\restriction}K_n$ is $\rho$-continuous because the topology of the compact space $K_n$ is generated by the metric $\rho$.

Since $Y$ is an absolute extensor, the $\rho$-continuous function $f{\restriction}K_0$ has a $\rho$-continuous extension $f_0:L_0\to Y$. By induction, for every $n\in\IN$  find a $\rho$-continuous function $f_n:L_n\to Y$ such that $f_n{\restriction}K_n=f{\restriction}K_n$ and $f_{n}{\restriction}L_{n-1}=f_{n-1}$. Such a function $f_n$ exists since $Y$ is an absolute extensor and $$f{\restriction}K_n\cap L_{n-1}=f{\restriction}K_{n-1}=f_{n-1}{\restriction}K_{n-1}=f_{n-1}{\restriction}K_n\cap L_{n-1}$$ (by the inductive assumption) and hence the function $(f{\restriction}K_n)\cup f_{n-1}$ is well-defined and $\rho$-continuous (so has a $\rho$-continuous extension $f_n$).

Now consider the function $\bar f:X\to Y$ such that $\bar f{\restriction}L_n=f_n$ for every $n\in\w$. Observe that for every $n\in\w$ we have $\bar f{\restriction}K_n=f_n{\restriction}K_n=f{\restriction}K_n$, which implies that $\bar f$ is an extension of the function $f$. Taking into account that the metric $\rho$ on $X$ is continuous, and the restrictions $\bar f{\restriction}L_n$, $n\in\w$, are $\rho$-continuous, we conclude that that these restrictions are continuous, which implies that the function $\bar f:X\to Y$ is $\bar\sigma$-continuous.

It is remains to show that the function $\bar f$ is linearly continuous. Fix $x\in X$. If $x\in X\setminus K$, then $x\in G_n\subset L_n$ for some $n\in\w$. Since  $\bar f|_{G_n}=f_n|_{G_n}$, the function $\bar f$ is $\rho$-continuous and hence continuous at $x$. If $x\in F$, then $x\in K_n$ for some $n\in\w$. Since $L_n$ is an $\ell$-neighborhood of $x$ and $\bar f{\restriction}L_n=f_n$ is continuous,  $\bar f$ is $\ell$-continuous at $x$.
\end{proof}

Now we can present {\em the proof  of Theorem~\ref{t:1.6}}.  Let $U$ be a strictly convex open set in a linear topological space $X$, $\partial U$ be the boundary of $U$ in $X$, and $K\subset \partial U$ be a $\sigma$-compact functionally closed set in $X$. Given any $F_\sigma$-measurable function $f:K\to Y$ to a Banach space $Y$, we need to find a linearly continuous $F_\sigma$-measurable extension $\bar f:X\to Y$ of $f$.

 If $K=\emptyset$, then the zero function $\bar f:X\to\{0\}\subset Y$ is a continuous function extending the function $f$. So, we assume that $K\ne\emptyset$ and then $\partial U\ne\emptyset$. By Lemma~\ref{l:subnorm}, the linear topological space $X$ admits a continuous norm. Then all compact subsets of $X$ are metrizable and second-countable. This implies that the $\sigma$-compact space $K$ has countable network and hence is hereditarily Lindel\"of and perfectly normal.

By Theorem 2.5 of \cite{BR19}, any Borel image of a Polish space has countable spread. Since  metrizable spaces with countable spread are separable \cite[4.1.15]{Engelking}, the image $f(K)$ is separable and hence is contained in a separable Banach subspace $Y'$ of the Banach space $Y$. By Dugundgi Theorem~\cite{Dugundji}, the closed unit ball $B=\{y\in Y':\|y\|\le 1\}$ of the Banach space $Y'$ is an absolute extensor.

By Lemma~\ref{lemFsigmaMeasurableFunctionAsSumOfSigmaContinuous}, the $F_\sigma$-measurable function $f:K\to Y'$ can be written as the sum of a uniformly convergent series $f=\sum_{n=0}^\infty f_n$ for some sequence  of $\bar\sigma$-continuous functions $f_n:K\to Y'$ such that $f_n(K)\subset \frac1{2^n}B$ for any $n>0$. By Lemma~\ref{lemExtentionOsSigmaContinuousFunction} for every $n\in\w$ there exists a linearly continuous $\bar\sigma$-continuous function $\bar f_n:X\to Y'$ such that $\bar f_n{\restriction}F=f_n$ and $\bar f_n(X)\subset \frac1{2^n}B$ if $n>0$. It follows that the series $\bar f:=\sum_{n=0}^\infty \bar f_n$ is uniformly convergent to a linearly continuous  function $\bar f:X\to Y'$ extending the $F_\sigma$-measurable function $f$. Since each $\bar\sigma$-continuous function $f_n$ is $F_\sigma$-measurable, the sum $\bar f$ of the uniformly convergent series $\sum_{n=0}^\infty f_n$ is $F_\sigma$-measurable by Theorem 2 in \mbox{\cite[\S 31]{Kuratowski1}}.

\bibliographystyle{amsplain}

\end{document}